\documentclass{article}
\usepackage{amsthm}
\usepackage{color}
\usepackage{enumerate}
\usepackage{amssymb,latexsym,amsxtra,amscd,amsfonts}
\usepackage[mathscr]{eucal}
\usepackage{dsfont}
\usepackage{calrsfs}

\numberwithin{equation}{section}

\newcommand{\theoref}[1]{Theorem~\ref{#1}}

\newcommand{\lemref}[1]{Lemma~\ref{#1}}
\newcommand{\prref}[1]{Proposition~\ref{#1}}

\makeatletter
\def\@biblabel#1{#1.}
\makeatother
\newcounter{obr}[section]

\newcounter{pvv}[section]
\renewcommand{\thepvv}{\thesection.\arabic{pvv}}
\newenvironment{pv}[2][]{\begin{trivlist}\refstepcounter{pvv}%
\item[\hspace{\labelsep}\normalfont\bfseries\thepvv. #2%
  \def\tmp{#1}\ifx\tmp\empty\else{} (#1)\fi.]}%
{\end{trivlist}}
\newenvironment{df}{\begin{pv}{Definition}}{\end{pv}}
\newenvironment{theo}[1][]{\begin{pv}[#1]{Theorem}\begin{itshape}}{\end{itshape}\end{pv}}
\newenvironment{pr}{\begin{pv}{Proposition}\begin{itshape}}{\end{itshape}\end{pv}}
\newenvironment{lem}{\begin{pv}{Lemma}\begin{itshape}}{\end{itshape}\end{pv}}

\newenvironment{exm}{\begin{pv}{Example}\begin{itshape}}{\end{itshape}\end{pv}}

\newcommand{\set}[2]{\ensuremath{\{ #1\,:\; #2\}}}

\newcommand{\cc}{\ensuremath{\mathbb C}}
\newcommand{\rr}{\ensuremath{\mathbb{R}}}

\newcommand{\8}{\ensuremath{\infty}}

\def\Ker{\mathop{\rm Ker}\nolimits}

\newcommand{\f}{\ensuremath{\varphi}}
 
\renewcommand{\l}{\ensuremath{\lambda}}

\newcommand{\vNa}{von Neumann algebra}

\newcommand{\Ca}{$C^\ast$-algebra}

\newcommand{\HS}{Hilbert space}

\newcommand{\ONB}{orthonormal basis}

\newcommand{\st}{such that}

\newcommand{\seq}[2]{\ensuremath{#1_1, \ldots, #1_{#2}}}

\newcommand{\nor}[1]{\ensuremath{\|#1\|}}

\newcommand{\J}[3]{\{#1,#2,#3\}}

\begin{document}
	
\begin{center}
{\bf \LARGE   Dye's Theorem for tripotents in von Neumann algebras and JBW$^*$-triples}\\[5mm]

{\Large Jan Hamhalter }\\[5mm]

Czech Technical University in Prague\\
Department of Mathematics\\
Faculty of Electrical Engineering\\
Technicka 2, 166 27 Prague 6\\
Czech Republic \\
hamhalte@fel.cvut.cz\\
\end{center}
\vspace{1cm}

{\small Abstract. We study morphisms of the generalized quantum logic of tripotents in JBW$^*$-triples and von Neumann algebras. Especially,  we establish generalization of celebrated Dye's  theorem on orthoisomorphisms between von Neumann  lattices   to this new  context.  We show one-to-one correspondence between maps on tripotents preserving orthogonality,  orthogonal suprema, and reflection $u\to -u$, on one side,  and    their extensions to  maps that are  real linear on  sets of  elements  with bounded range tripotents on  the other side.   In a more general  description we show that quantum logic morphisms on tripotent structures   are given by a family of Jordan *-homomorphisms on 2-Peirce subspaces.  By examples we exhibit new phenomena for tripotent morphisms that have no analogy for projection lattices and demonstrated that  the above mention tripotent versions of  Dye's theorem cannot be improved.  On the other hand,  in a  special case of  JBW$^*$-algebras we can  generalize Dye's result directly. Besides we show that structure of tripotents in C$^*$-algebras determines their  projection poset and is a complete Jordan invariant for \vNa s.     }\\

key words: tripotent poset, quantum logic morphisms,  JBW$^*$ algebras, JBW$^*$-triples.         \\


\section{Introduction}

     The aim of the present paper is to explore the concept of orthogonality and order in structures associated with \vNa s and Jordan triples. Especially, we study morphisms of projection lattices of JBW$^*$-algebras and generalized quantum logics of tripotents in JBW$^*$-triples. The former structure is going back to John  von Neumann's  work on continuous geometry \cite{vN-projective geometry},  foundations of operator algebras  \cite{Kadison},  and  foundations of quantum theory \cite{Neumann-F}. The latter structure  is younger and stems from infinite dimensional holomorphy, Jordan theory and its applications to  mathematical physics \cite{Chu,Friedmann,Upmeier}.  Main results of the paper concern generalizations of celebrated Day's theorem \cite{Dye} for the above indicated order  structures showing their intimate connection with their underlying  linear structure. 
       
        Dye's theorem is a generalization of famous Wigner's theorem about symmetries of quantum system   \cite{Wigner}. Ulhorn's logic theoretic version \cite{Ulhorn}  of this principle  states that orthogonality relation in the set  $P(H)$ of projections acting on a \HS{,} $H$ with $\dim H\ge 3$, determines dynamics of the system. More precisely, if $\f:P(H)\to P(H)$ is  a bijection preserving orthogonality in both directions, i.e. 
        \[ pq=0   \Longleftrightarrow \f(p)\f(q)=0\,,              \] 
        then there is a unitary or antiunitary operator $u$ acting on $H$ \st{}
        \[      \f(p)= u^*pu\,, \qquad \text{ for all } p\in P(H).                             \]
     
      In unifying  reformulation avoiding distinguishing between unitary and antiunitary case,  we can say that there is a Jordan *-isomorphism $J: B(H)\to B(H)$, where $B(H)$ is the algebra of all bounded operators acting on $H$,     \st{}  
      \[     \f(p)= J(p)\,, \qquad \text{ for all } p\in P(H).                \]   
    Following advances made by John von Neumann in his project of  continuous geometry \cite{vN-projective geometry},  Day tackled the problem  of describing orthoisomorphisms between projection lattices of \vNa s{} which involves Ulhron's result as a very special case. His principal result, known today  as Dye's theorem,  reads as follows
    
    \begin{theo}{\bf (Dye's Theorem)}\label{DT} 
    	Let $M$ and $N$ be \vNa s, where $M$ does not have Type $I_2$ direct summand.  Let $\f:P(M)\to P(N)$ be orthoisomorphism between projection lattices of $M$ and $N$, respectively. Then there is a Jordan *-isomorphism $J: M\to N$ extending $\f$.     	
    \end{theo}            

The assumption on absence of Type $I_2$ direct summand corresponds to assumption $\dim H\ge 3$ in Ulhorn's version of  Wigner's theorem. Dye's theorem shows that orthogonality relation in the projection lattice of a \vNa{} determines its Jordan *- structure given by the anticommutatnt product $(x,y)\to \frac 12 (xy+yx)$ and the standard *-operation. This  result was then generalized beyond \vNa s{} by the author in \cite{Ham-Dye}.  \\       

The main message of Dye's theorem is the fact that any map preserving orthogonality relation can be "linearised"  by finding its {\em linear} extension. Once we have linearity,  the fact that the  extension has to be Jordan *-isomorphism can be  deduced relatively easily  as a consequence of preserving projections and their orthogonality. L.J.Bunce and J.D.M.Wright fully realised this fact in their  work \cite{Bunce3} that was a turning point for  further investigation. Namely, they showed that Gleason's theorem on extending  probability measure on projections  to a linear  map can be applied to generalizing Dye's theorem in a few directions. At first it enabled to show a version of Dye's theorem for non-bijective maps between projections. Moreover,  they established  validity of Dye's theorem for projection structures in JW-algebras that are more general than \vNa s. (JW-algebra is  a weakly closed subspace of the real space of self-adjoint operators acting on a \HS{} that is closed    under forming the  squares $a\to a^2$.)  We continue this research by showing Dye's theorem for even more general JBW$^*$-algebras.  Let $A$ and $B$ be JBW$^*$-algebras, where $A$ does not have Type $I_2$ part.   We show that any map $\f: P(A)\to P(B)$   between projection lattices  of  JBW$^*$-algebras  such that \f{} preserves zero, orthogonality and suprema of orthogonal    projection (such a map will be called  quantum logic morphism) is a restriction of a Jordan *-homomorphism $J:A\to B$.  \\

  Principal contribution  of our paper concerns Dye's theorem in the context of JBW$^*$-triples. JBW$^*$-triples are Banach spaces that generalize \vNa s. However, they involve also Hilbert spaces and structures  of operators between different Hilbert spaces (rectangular matrices), which are seldom \vNa s. The theory of JBW$^*$-triples has been developing rapidly  recently and is vital for infinite dimensional complex analysis, differential geometry and mathematical physics.              Projection lattice is a  prominent order structure associated with classifications  of \vNa s. There is another poset, called tripotent poset,  that underlies theory of JBW$^*$-triples in a similar way.  An element $u$ in a \vNa{} $M$ is a tripotent if it is partial isometry, i.e. if  $u=uu^*u$. This class of operators involves unitary operators as well as projections. Order is defined as follows 
    \[ u\le v \qquad \Leftrightarrow\qquad u=uv^*u\,.    \]
The resulting poset $(U(M), \le)$ of all tripotents in $M$ includes projection lattice $P(M)$ as a principal ideal. Besides, there is a natural orthogonality relation on $U(M)$. The tripotents  $u$ and $v$ are orthogonal  if $uu^*v=0$. It it the same as to say that  kernels of $u$ and $v$ are orthogonal subspaces in the underlying \HS{} and the same holds for ranges of $u$ and $v$. The importance of the structure $(U(M), \le, \perp      )$ was recognized by  Edwards and R\"{u}timann who studied tripotent posets from the perspectives of orthomodular structures and obtained many deep results between functional analysis and theory of quantum logics \cite{Edwards-Rut,Edwards-Rut2}. Especially, they showed that tripotent poset of a JBW$^*$-triple is Dedekind complete  (i.e. each upper bounded set has supremum) and found connection of this poset with facial geometry of the unit ball. Unlike von Neumann projection lattice,       triple poset is not upward directed and has many maximal elements. However,  it can be  organized into generalized quantum logic (nonunital version of orthomodular poset).        In further development a  fruitful interplay of the theory of orthomodular posets and JBW$^*$-triples initiated by     Edwards and R\"{u}timann seemed to be  neglected until recent paper \cite{HKP}. We believe that it a pity in the light of recent development in JBW$^*$-triple theory as well in foundations of quantum theory \cite{Landsman}. For this reason we would like to revive this line of the research by studying morphisms between tripotent posets and appropriate forms of Dye's theorem for these structures.\\        

In  the beginning of our research,   we realised that one  cannot generalise Dye's theorem verbatim to Jordan triple structures. Indeed, easy counterexamples show that there is an orthoisomorphism  between tripotents that does not  preserve order  and cannot be so extended to a Jordan morphism. That is why we have to consider quantum logic morphisms (see \cite{Bunce3})), i.e maps  preserving suprema of orthogonal elements,  in addition to preserving orthogonality. Further, our counterexamples show that even if we exclude Type $I_2$ part as in original  version of Dye's theorem, the theorem is not valid. In other words, quantum logic morphisms are always more general than restrictions of linear Jordan maps. However,  we show that any quantum logic morphism of triple structures   extends to a map that is additive  on relevant  parts. Most importantly,   we establish one-to-one correspondence between quantum logic morphisms preserving  reflection $u\to -u$   and real homogeneous  maps preserving tripotents that are additive with respect to elements  whose range tripotents have upper bound.  Such maps are called local Jordan morphisms.    Another approach how to describe  all quantum    logic morphisms elaborated in this paper is based on a family of linear Jordan maps that are defined on homotopes and are consistent in some sense. Beside generalization of Dye's theorem we also bring some new results on tripotent posets and their morphisms. \\

Tripotent poset constitute invariant in the category of JB$^*$-triples. Even if it is  not in the main focus of our paper, we also discuss briefly a  natural question of whether it is a complete invariant in case of JBW$^*$-triples. We show that the answer is in the positive in case of triple posets in \vNa s algebras.  More specifically, we prove that the following conditions are equivalent:  (i) \vNa s $M$ and $N$ are Jordan *-isomorphic (2) tripotent posets $U(M)$ and $U(N)$ are isomorphic as generalised quantum logics (iii)  projection lattices $P(M)$ and $P(N)$ are isomorphic as quantum logics. We also show that the tripotent posets of two C$^*$-algebras are isomorphic if and only if  their projection posets are isomorphic.    As a consequence, tripotent posets  carry the same amount of information as projection lattices. This seems to be interesting because tripotent poets  include also non normal elements and are much larger than projection posets \\

    Our paper is organised as follows. After introduction and recalling basic notions we introduce orthomodular order structures in the second section and prove some easy statements needed later. In the third part we generalise Dye's theorem to projection lattices of  JBW$^*$ algebras.  In the forth  section of the present note we focus on ordered tripotent structures and their morphisms. We exhibit some examples showing that our later results are optimal. Sections 5 and 6 contain main results describing quantum logic morphisms between tripotent  structures in terms of families of Jordan maps and local Jordan maps. Concluding section  is an invitation to further research and contains    discussion on complete order invariant of von Neumann algebras. \\

Let us now recall a few concepts and fix the notation. For the theory of \Ca s and \vNa s  the reader is referred  to monographs \cite{Kadison,Tak}. For fundamentals of the theory of Jordan algebras and  Jordan triple systems we recommend  monograph \cite{Chu,Garcia,Hanche,Upmeier,Topping}.\\

Given a normed space $X$, $B_1(X)$  shall denote its unit ball. By the symbol $B(H)$ we shall denote the algebra of all bounded operators on a complex \HS{} $H$. 
Jordan algebra is a real or complex commutative algebra endowed with a Jordan product, $\circ$,  satisfying the identity $x\circ (x^2\circ y)= x^2\circ (x\circ y)$.   By a JB-algebra we mean a real Banach space that is simultaneously a (real) Jordan algebra, for which we have $\nor{x\circ y}\le  \nor{x}\cdot \nor{y}$,   $\nor{x^2}=\nor{x}^2$ and $\nor{x^2-y^2}\le \max\{\nor{x^2}, \nor{y^2}\}$. By a JC-algebra we understand  a real closed subalgebra of the self-adjoint part of  $B(H)$ that is closed under the squares and whose Jordan product is $x\circ y=\frac 12 (xy+yx)$. 
Let $A$ be a JB-algebra. Positive elements in $A$ are  elements of the form $a^2$,  $a\in A$.   
Elements $x,y\in A$ are said to operator commute if $x\circ (y\circ z)= y\circ (x\circ z)$ for all $z\in A$. The center $Z(A)$ of $A$ is the set of all elements operator commuting with each element of $A$. 
A JW-algebra is a JC-algebra that is moreover closed in the weak operator  topology on $B(H)$. JBW-algebra is a JB-algebra that has a (unique) predual.
Given $a\in A$ we shall define an operator $U_a$  acting on $A$ by $U_a(x)= 2 a\circ (a\circ x) - a^2\circ x.$  Let us remark that operator $U_a$ is positive  in the sense that leaves the positive cone of $A$ invariant. 


Let $A$ be a complex Jordan algebra endowed with an involution $\ast$. Then one can define the   triple product, $\circ$, on  $A$ by

\begin{equation}\label{*} \{a, b, c\}= a\circ (b^*\circ c)+ c\circ (a\circ b^*)-b^*\circ (a\circ b)\,.
\end{equation}

Let $A$ be JC$^*$-algebra, that is a closed complex Jordan subalgebra of $B(H)$ which  is invariant with respect to  adjoints and  endowed with the Jordan product $a\circ b = \frac 12 (ab+ba)$. Then we have that 
\[   \J abc = \frac 12 (ab^*c+cb^*a)       \,.              \]

JB$^*$-algebra is a  Jordan *-algebra that is simultaneously a Banach space whose norm satisfies:
\[  \nor{a^*}=\nor{a}\,, \qquad \nor{a\circ b}\le \nor a \nor b\,, \qquad  \nor{\{a,a,a\}}=\nor a^3\,.\] 
The self-adjoint part of a JB$^*$-algebra is  the set $H(A)=\set{a\in A}{a=a^*}$.  It is a JB-algebra and each JB-algebra can be  obtained in this manner \cite{Wright}. If a JB$^*$-algebra admits a predual, then it is called a JBW$^*$-algebra.  A projection in a JB$^*$-algebra (resp. JB-algebra)  is a self-adjoint idempotent (resp.  idempotent). The set of projections in a JB$^*$-algebra  or in a JB-algebra $A$ will be denoted $P(A)$.  A linear functional $f$  on a JB$^*$-algebra $A$ is called positive if it takes positive values on elements in the positive part of $A$, that is if $f(a^2)\ge 0$ for all $a\in H(A)$. If $f$ is moreover norm one, then it is called state. \\

A Jordan triple is a complex space $E$ endowed with triple product $(a,b,c)\to \{a, b, c\}$ which is symmetric and linear in the first and the third variable and conjugate linear in the second variable and satisfies the identity

\[ [L(a,b), L(c,d)]= L(\{a,b,c\}, d)-L(c, \{d, a, b\})=\] \[ =L(a, \{b,c,d\})-L(\{c,d,a\}, b)\,,\] 
where $[\cdot,\cdot]$ denotes the commutator and $L$ is the  mapping from $E\times E$ into the space of linear operators on $E$ defined by $L(a,b)c=\{a,b,c\}$. A Jordan  triple $E$ is said to be a JB$^*$-triple if the following holds:
\begin{itemize}
	\item $E$ is a Banach space and $L$ is a continuous map from $E\times E$ into the space  $B(E)$
	of bounded operators acting  on $E$.
	\item For each $a\in E$, $L(a,a)$ is a hermitian operator with nonnegative spectrum and satisfies $\nor{L(a,a)}=\nor{a}^2$. 
	(Let us recall that a bounded  operator $T$ acting on some complex Banach space  is called hermitian if $\nor{ e^{itT}} =1$ for all $t\in \rr$. )
\end{itemize}

The JBW$^*$-triple is a JB$^*$-triple that is a dual Banach space. Any JBW$^*$-algebra  endowed with the triple product (\ref{*})  is  a  JBW$^*$-triple.\\

Tripotent $u$ in a JB$^*$-triple $E$ is  an element  satisfying $\{u,u,u\}=u$.  The set of all tripotents of $E$ will be denoted by $U(E)$. 
Each tripotent $u$ is responsible for decomposition of  $E$ into closed subspaces 
\[ E = E_0(u)\oplus E_1(u)\oplus E_2(u)\,,\]
where $E_i(u)$ is the eingenspace of $L(u,u)$ corresponding to the  eigenvalue $\frac i2$. 
A complete tripotent is a tripotent $u\in E$ for which $E_0(u)=\{0\}$. 
It is known that $x$ is an  extreme point of the unit ball of a JBW$^*$-triple if and only if it is a complete tripotent. 
The space   $E_2(u)$ can be made into  JBW$^*$-algebra with respect to the following involution $*_u$   and  Jordan product $\circ_u$:
\[ x\circ_u y = \{x, u,y\}\,, \qquad x^{*_u}=\{ u, x,u\} \qquad x,y\in E_2(u)\,.\]
The tripotent $u$ is the unit in the JB$^*$-algebra $E_2(u)$.
The triple product induced by the Jordan product $\circ_u$  via (\ref{*})  coincides with the original triple product  restricted to $E_2(u)$. We shall denote the JB$^*$-algebra defined in this way by the symbol  $E(u)$. Sometimes  $E(u)$  is called homotope of $E$ corresponding to $u$. Tripotent $u$ in $E$ is called unitary if $E_2(u)=E$. In  that case the triple product on $E$  is coming  from the underlying JB$^*$-algebra $E(u)$.           \\

  Let $H$ and $K$ be Hilbert spaces and $B(H,K)$ the space of all bounded operators  from $H$ to $K$. A J$^*$-algebra is a closed subset of $B(H,K)$ which is  closed under the product $a\to aa^*a$. When endowed with the triple product $\J abc = \frac 12 (ab^*c+cb^*a)$,  J$^*$-algebra is a JB$^*$-triple.  
   If $E$ is a J$^*$-algebra, then $u\in E$ is a tripotent if and only if it is a partial isometry, that is an element $u$ for which $uu^*$ and $u^*u$ are projections (in the corresponding spaces). We shall define initial projection of a tripotent  $u$ as  $p_i(u)= u^*u$ and final projection of $u$ by $p_f(u)= uu^*$. In this case we have 
   \[  E_2(u)= p_f(u)Ep_i(u)\,.                \]
  If a  J$^*$-algebra $A$  is closed in the weak*-topology,     then it is a JBW$^*$-triple.   An example of a J$^*$-algebra is a JC$^*$-algebra.   An example of a J$^*$-algebra that may not be a  JB$^*$-algebra is the JBW$^*$-triple $B(H)_a$ of all antisymmetric operators in $B(H)$. Let us recall  that an operator $X\in B(H)$ is antisymmetric if $X^t=-X^t$, where $X\to X^t$ is the  transpose operation with respect to a fixed \ONB{} of $H$. It is known that  this JB$^*$-triple system is a JB$^*$-algebra if and only if $H$  does not  have  odd finite dimension.\\
  
  Let $x$ be a nonzero  element in JBW-algebra $M$. Its range projection $p$ is the smallest projection in $M$ \st{} 
  $p\circ x = x$. This projection is always contained in a JBW-subalgebra of $M$ generated by $x$.   
  For each  norm one  element $x$ in a JBW$^*$-triple $E$, $r(x)$ will denote its     range tripotent. It is the smallest tripotent  $e$ in $M$ for which   $x$ is a positive element in $E(e)$.   If $x$ is a general nonzero element then its range tripotent is the range tripotent of $\frac x{\nor x}.$ We  set $r(0)=0$. The range tripotent  is always contained in a the JBW$^*$-subtriple of $E$ generated by $x$. Suppose that $x$ is a positive element in $E(u)$  for some tripotent $u\in E$.  Then its range tripotent coincides with its range projection in $E(u)$.\\
  
   Let $(A,\circ)$ and $(B,\circ)$ be JB$^*$-algebras. A linear map $J:A\to B$ is called a Jordan *-homomorphism if $J(a\circ b)=J(a)\circ J(b)$ and $J(a^*)=J(a)^*$ for all $a,b\in A$.  It is called a  Jordan *-isomorphism if it is a bijective Jordan *-homomorphism. Jordan homomorphism $J:(A,\circ)\to (B,\circ) $ between JB-algebras $A$ and $B$ is a map preserving product, that is $J(a\circ b)=J(a)\circ J(b)$. If $J$ is  is bijective it is called a Jordan isomorphism. A linear map $J:E\to F$ between JB$^*$-triples $E$ and $F$ is called a Jordan triple homomorphisms if it preserves triple product, that is $J\J abc =\J {J(a)}{J(b)}{J(c)}$. If a Jordan triple homomorphism is a bijection, then we are talking about Jordan triple isomorphism. Celebrated Kaup's theorem  assures that a surjective linear operator between JB$^*$-triples is an isometry if and only if it is  a triple isomorphism.  \\



	It is well known that Dye's theorem does not hold for algebra of two by two matrices. In fact,  it does not hold for JBW-algebras of Type $I_2$ which is much wider class. We do not give original definition of these algebras since we are not going to use it. However,  we describe what Type $I_2$ means for readers not familiar with classification theory of Jordan algebras.  Let $H_n$, where $n<\8$,  be an $n$-dimensional real \HS. We define JB-algebra $V_n=H_n\oplus  \rr$ as a Banach space with norm $\nor{x\oplus \l 1     }=\nor x +|\l|$ and with  multiplication
	\[ (a+\l 1)\circ (b+\mu 1)= (\mu a+\l b)\oplus (\langle a, b \rangle +\l\mu )    \,.                            \] 
	Such  an algebra is called (finite dimensional) spin factor. Let $C(X, V_n)$ be   the JBW-algebra of continuous functions from a hyperstonean space $X$ into spin factor $V_n$ with pointwise defined Jordan multiplications  and maximum norm.  Type $I_2$ algebra is isomorphic to a nonzero direct sum
	\[  \sum_{k=0}^\8  A_{n_k}  \,,         \]
	where $(n_k)$ is a strictly increasing sequence of integers and each $A_{n_k}$ is either zero or  the algebra $C(X_k, V_{n_k})$, where       $X_{n_k}$ is a hyperstonean space.  
		We shall say that a JBW-algebra is regular if it does not contain any  direct summand  of Type $I_2$. A JBW$^*$-triple $E$ is said to be regular if $H(E(u))$ is a  regular JBW-algebra for each complete tripotent $u\in E$. If $M$ is a \vNa{},  then  it is regular as a JBW$^*$-triple if and only if it does not contain any Type $I_2$ direct summand. Indeed, having a complete tripotent $u$ in a \vNa{} $M$,  we shall show in the proof of  \lemref{Ev1} that there is a unital Jordan triple isomorphisms between  $M=M(1)$ and $M(u)$.  This isomorphism  is  a Jordan *- isomorphism. Any Jordan *-isomorphism preserves Type $I_2$ direct summands.  Therefore $M$ does not have any Type $I_2$ direct summand if and only if the same holds for $M(u)$.

Jordan triple isomorphism $\Phi: E\to F$. \section{Order structures}

In this section we gather standard definitions and notations concerning ordered structures. Let $(P, \le)$ be a partially ordered set (poset in short). Given $a,b\in P$ we shall denote $[a,b]=\set{x}{a\le x\le b}$. This set will called the interval.  By $a\land b$ and $a\lor b$ we shall mean the join (infimum) and meet (supremum)  of the set $\{a,b\}$, respectively. In case of general set $S\subset P$, we shall denote  the respective meet and join by $\bigwedge S$ and $\bigvee S$. A subset of a poset is called bounded if it has lower and upper bound (that is, if it is a subset of some interval). Further, a subset $S\subset P$  is called upper bounded if it has  an upper bound.  By a conditionally complete poset we mean a poset for which every upper bounded nonempty set has supremum.  
An upward directed poset is a poset in which every two point set has an upper bound. Now we shall recall basic concepts of morphisms between posets.

\begin{df}
	Let $P$ and $Q$ be posets, and $\f:P\to Q$ a map. Then \f{} is called 
	\begin{itemize}
		\item an order  morphism  if $a\le b $ implies $\f(a)\le \f(b)$ for each $a,b\in P$.  We also say that \f{} preserves order in one direction; 
		\item an embedding of posets if $a\le b$ if and only if $\f(a)\le \f(b)$ for each $a,b\in P$. We also say that \f{} preserves order in both directions. Order embedding is always  an injective map;
		\item an order isomorphism if it is a surjective order embedding.
		
	\end{itemize}
\end{df}

The same terminology will be employed in case of order reversing maps. For example, a map $\f:P\to Q$ is an order antimorphism if  $a\le  b$ implies $\f(a)\ge \f(b)$ in $Q$.

We shall be mainly interested in poset endowed with some concept of orthogonality or orthocomplementation. 

\begin{df}\label{kometa}
	Let $P$ be poset with a least element 0. Orthogonality relation, $\perp$,  on $P$ is a relation satisfying the following conditions:
	
	\begin{enumerate}
		\item $\perp$ is a symmetric relation.
		\item $0\perp a$ for each $a\in P$. 
		\item $a\perp a$ implies $a=0$.
		\item If $a\perp b$, then $c\perp b$ whenever $c\le a$. 
	\end{enumerate}
	
\end{df}

\begin{df}
	Let $P$ be a poset with a least element 0 and a greatest element 1.
	
	\begin{itemize}
		\item 
		$P$ is called orthoposet if there is an operation $a\to a^\perp$ on $P$, called orthocomplementation,  fulfilling the following conditions for each $a,b\in P$: 
		\begin{enumerate}
			\item $a\le b$ implies $b^\perp \le a^\perp$
			\item $a^{\perp\perp}=a$
			\item $a\land a^\perp =0$  and $a\lor a^{\perp}=1$.
		\end{enumerate}   		
		We say that two elements $a,b$ in an orthoposet are orthogonal, written $a\perp b$,  if $a\le b^\perp$. Let us remark that      $\perp$   is an  orthogonality relation on $P$ in the sense of Definition~\ref{kometa}. 
		
		\item   An orthoposet $(P,\le, \perp)$ is called an orthomodular poset or quantum logic  if
		\[  a\lor b \text{ exists whenever } a\perp b               \]
		and  the following orthomodular law is satisfied:
		\[  b = a \lor  (b\land a^\perp)\,,     \]
		whenever $a\le b$. An orthomodular lattice is an orthomodular poset that is a  lattice.
		\item 	A generalized orthomodular poset $P$  (or a generalized quantum logic) is a poset with a least element 0, such that each interval  $[0, a]$, $a\in P$ is an orthoposet
		endowed with orthocomplemantation $ x\to x^{\perp_a} $ \st{}  the following conditions hold:
		\begin{enumerate}
			\item $([0,a], \le, \perp_a)$ is a quantum logic for each $a\in P$.
			\item If $a\le b$,  then $x^{\perp_a}=x^{\perp_b}\land a$ for all $x\in [0,a]$.  
		\end{enumerate}

		
	\end{itemize}

\end{df}

Let us remark that if $P$ is a quantum logic,  then it is canonically a generalized quantum logic with local orthocomplementation  on each interval $[0,a]$ given by 
$$ b^{\perp_a} = a\land b^\perp \qquad b\le a\,. $$  Therefore  quantum logics can be viewed as unital generalised quantum logics. Further, having generalized quantum logic $P$,   we can induce  a canonical orthogonality relation,  $\perp$, on $P$ by setting $a\perp b$ if there is a $v\in P$ with $v\ge a,b$ \st{} $a$ and $b$   are orthogonal in quantum logic $[0,v]$, that is $b\le a^{\perp_v}$.  It can be shown that this definition  does not depend on $v$.   A typical example of a generalized quantum logic is the poset $P(A)$ of projection in a (possibly nonunital) \Ca{} $A$. Indeed, for each projection $p\in A$ we introduce orthocomplementation $\perp_p$ on $[0,p]$ by  setting
\[ q^{\perp_p}=p-q\,, \]                  
for $q\le p$. 

\begin{df}
	Let $P$ and $Q$ be posets endowed with  relation of orthogonality, and \f{} a map $\f:P\to Q$.
	\begin{itemize}
		\item  \f{} is called orthomorphism if it preserves orthogonality relation in one direction, that is  if for each  $a,b\in P$
		\[  \f(a)\perp \f(b)  \text{ whenever } a\perp b\,.          \]	
		\item \f{} is called orthoisomorphism if  it is a bijection  preserving  orthogonality relation in both directions, that is if
		\[  \f(a)\perp \f(b)  \text{ if and only if  } a\perp b\,.         \] 
	\end{itemize}
\end{df}

\begin{df}
	Let $P$ and $Q$ be generalized orthomodular posets and $\f:P\to Q$.
	Then \f{} is called a quantum logic morphism if for each orthogonal $a,b\in P$ we have 
	\begin{enumerate}
		\item \f(0)=0
		\item $\f(a)\perp \f(b)$
		\item $\f(a\lor b)= \f(a)\lor \f(b)$.
	\end{enumerate}
	If \f{} is a bijection \st{} both \f{} an $\f^{-1}$ are quantum logic morphisms, then \f{} is called quantum logic isomorphism.   
\end{df}



We have the following simple observation.

\begin{pr}
	Let $\f: P\to Q$ be a quantum logic morphism between generalized orthomodular posets $P$ and $Q$. Then \f{} is an orthomorphism and order morphism.
\end{pr}

\begin{proof}
	By definition \f{} preserves orthogonality. Take $a\le b$ in $P$. Hence $b= a\lor a^{\perp_b}$. Then $ \f(b)= \f(a)\lor \f(a^{\perp_b})$. It says  that $\f(a)\le \f(b)$.

\end{proof}

On the other hand, in  the unital case we see that orthoisomorphism is a quantum logic isomorphism.

\begin{pr}\label{poset-1}
	Let $P$ and $Q$ be quantum logics  and $\f:P\to Q$  an orthoisomorphism.  Then \f{} is an order isomorphism preserving orthocomplements. 
\end{pr}

\begin{proof}
	As 0 is the only element that is orthogonal to every element of the poset,  we see that $\f(0)=0$. Further,  as  the largest element  1 can be characterized as the only element that is orthogonal only to 0, we can conclude that  $\f(1)=1$. Let $a\in P$. Then $\f(a), \f(a^\perp)$ are orthogonal.  Put $b=\f(a)\lor \f(a^\perp)$ . Then $\f^{-1}(b^\perp)$  is orthogonal to both $a$ and $a^\perp$. Therefore $\f^{-1}(b^\perp)\perp 1$ and so it is zero. Consequently, $b=\f(a)\lor \f(a^\perp)=1$. Put $c=\f(a^\perp)$. Then $c\perp \f(a)$ and $\f(a)\lor c=1$. As $c\le \f(a)^\perp$ there is,  by the orthomodular law,  $d\in Q$ with $c\perp d$ and $c\lor d=\f(a)^\perp$. Therefore $d\perp \f(a)$ and so $d\perp (c\lor \f(a))=1$. Hence $d=0$ and so $\f(a^\perp)= \f(a)^\perp$.   We have shown that \f{} preserves orthocomplementation. As $a\le b$ in $P$ if and only if $a\perp b^\perp$ and \f{} preserves orthocomplements and orthogonality, we have that \f{} preserves the order. It is also clear from symmetry that \f{} preserves the order in both directions.   
	This completes the proof. 
\end{proof}

\section{Projection lattices}

In this part we shall generalize  Dye's theorem to  projection lattices of JBW$^*$-algebras. For each JB$^*$-algebra $(A,\circ)$ we can associate its projection poset  $P(A)$, where order of projections is given by $p\le q$ if $p\circ  q=p$.  In this case $p+q=p\lor q$. If $A$ is unital, then  $P(A)$ becomes an   orthomodular poset with orthocomplementation $p\to 1-p$ . (If $A$ is nonunital, then $P(A)$ can be endowed with the structure of generalized orthomodular poset, but we shall not need this level of abstraction.)
We can define  orthogonality relation on $P(A)$ by $p\perp q$ if $p\circ  q=0$. If $A$ is unital, then orthogonality relation on $P(A)$ is the one induced  canonically by the   orthocomplementation, i.e. $p\perp q$ if $p\le 1-q$.             As a consequence of   \prref{poset-1} we have that any orthoisomorphism between projection structures of unital JB$^*$-algebras is   order isomorphisms preserving orthocomplements.  It is well known that if $A$ is a JBW$^*$-algebra,  then $P(A)$ is a complete lattice.  \\

In  \theoref{Dye-Jordan algebras} we  present
non-bijective version of Dye's theorem for JBW$^*$-algebras. Our sharpest weapon will be 
deep Gleason's theorem for Jordan algebras proved by Bunce and Wright in \cite{BW1,BW2}.  
Let us recall that positive finitely additive  measure on the projection  poset $P(A)$ of a JB$^*$- algebra $A$ is a  map $\varrho$  from $P(A)$ into an  interval $[0,\8]$ \st{}    $\varrho(p+q)=\varrho(p)+\varrho(q)$ whenever $p\perp q$.\\

\begin{theo}{\bf (Jordan version  of Gleason's Theorem)}\label{GT}
	
	Let $W$ be a  JBW-algebra \st{} $W$ does not contain  any Type $I_2$ direct summand. Let $\varrho$ be a finitely additive positive measure on $P(W)$. Then $\varrho$ extends to a unique positive functional on $W$.

\end{theo}

The following theorem has been proved in \cite{Bunce3} for JW-algebras.

\begin{theo}
	\label{Dye-Jordan algebras}
	Let $A$ and $B$  be  JBW$^{*}$-algebras \st{}  $H(A)$      does not contain  any Type $I_2$ direct summand.    Suppose that $\f: P(A)\to P(B)$ is a quantum logic morphism. 	Then $\f$  extends uniquely  to a Jordan *-homomorphism 	$ J: A\to B\,.$     
\end{theo}

\begin{proof} We shall follow ideas of the proof in \cite{Bunce3}.    
	First let us note that any Jordan homomorphism between self-adjoint parts of JBW$^*$-algebras (viewed as JBW-algebras) can be canonically extended to a Jordan *-homomorphism between whole algebras. Therefore we restrict  ourselves to  JBW-algebras $M=H(A)$ and $N=H(B)$ and show that $\f$ extends to a Jordan homomorphism $J:M\to N$.   
	If  $p$ and $q$ are orthogonal projections,  then their supremum is their sum $p+q$. Hence, by the property of quantum logic morphism  we have that $\f(p+q)=\f(p)+\f(q)$ whenever $p\perp q$.\\

	Let us take a positive functional $f$  on $N$. Composition $f\circ \f$ is a positive  finitely additive measure on  $P(A)$. According to \theoref{GT}, we have that 
	there is a positive functional $\hat{f}$  on $M$ that extends $f\circ \f$. Let $L(M)$ be the linear span of $P(A)$ in $M$. Pick up $x\in L(M)$ and suppose that we have two expressions of $x$ as linear combinations of projections, say
	\[  x=\sum_{i=1}^n \l_ip_i= \sum_{j=1}^m \mu_j q_j\,,                          \]
	where $\seq \l n, \ \seq \mu m \in \rr$ and $\seq pn\,, \  \seq qm$ are  projections. Then   we have that 
	\[ \hat{f}(x) = \sum_{i=1}^n \l_i \hat{f}(p_i) =     \sum_{j=1}^m \mu_j \hat{f}(q_j)\,.      
	\]    
	In other words, 
	\[   f(\sum_{i=1}^m \l_i\f(p_i))= f(\sum_{j=1}^m \mu_j \mu(q_j))\,.\]                        
	By the Hahn Banach Theorem  and the fact that dual of $N$ is spanned by positive functionals,      we infer that 
	\[    \sum_{i=1}^m \l_i\f(p_i)=       \sum_{j=1}^m \mu_j \f(q_j)\,.                \]
	This allows us  to define a linear map \[T: L(M)\to L(N): \sum_{i=1}^n \l_ip_i\to \sum_{i=1}^m \l_i\, \mu(p_i)\,.        \] 
	
	We shall show that this map is positive. For a contradiction suppose that $\sum_{i=1}^n \l_ip_i$ is positive while $\sum_{i=1}^n \l_i\f(p_i)$ is not positive.  In this case there is a positive functional $f$ on $N$ \st{} $\sum_{i=1}^n \l_i f(\mu(p_i))<0$. Therefore $\hat{f}$ is a positive functional on $M$ with 
	$\hat{f}   (\sum_{i=1}^n \l_ip_i)<0$,  but this is a contradiction. Positivity of  $T$ implies its boundedness. Indeed, it follows from the inequality $\nor{ Tx}\le \nor {T(1)}$ whenever $x\in L(M)$ is a positive  element with norm less then one.             Consequently, $T$ can be extended   to a bounded linear map  (denoted again by $T$) from $M$ to $N$. Finally,  as $T$ preserves projections,  it is  a Jordan homomorphism. This fact was shown in Theorem A.4 in \cite{Molnar} for \vNa s,  but the proof for JBW$^*$-algebras  is the same. 
	
\end{proof}

Now we shall consider bijective variant  of the previous result, which is a  direct generalization of Dye's theorem for \vNa s.

\begin{theo}\label{Dye-Jordan injective}
	Let $A$ be a JBW$^{*}$-algebra \st{} $H(A)$  does not contain  any Type $I_2$ direct summand  and  $B$ is  another JBW$^*$-algebra. Suppose that $\f: P(A)\to P(B)$  is an orthoisomorphism. Then $\f$ extends to a Jordan  *-isomorphism $J:A\to B$.  
\end{theo}

\begin{proof}
	By \theoref{Dye-Jordan algebras},  
	$\f$ extends to a Jordan $*$-homomorphism $J:A\to B$.   The image $J(P(A))$ contains $P(B)$.
	Since any image of a Jordan*-homomorphism of a JB$^*$-algebra is closed and the closed linear   span of $P(B)$ is dense in $B$,  we have that $J(A)=B$. 	
	Let us demonstrate that  $J$ is injective. We know that
	$\Ker J$ is a JB$^*$-algebra and so it is linearly  generated by positive elements. So if $\Ker J\not=\{0\}$,  then there is a  positive norm one element $x\in J$. Now we can continue in the same way as  in the proof of   \cite[Theorem 8.1.2 p. 256]{qmt}. By the spectral theory we can write  
	\[  x= \sum_n \dfrac{1		}{2^n}\, p_n\,,                        \]
	where $p_n$'s are mutually commuting projections in $H(A)$. As $0=J(x)\ge \frac 1{2^n} J(p_n)$, we have that  $J(p_n)=0$ for all $n$. However, at leat one $p_n$ has to be nonzero. This is a contradiction with injectivity of $\f$ on the projection lattice.

\end{proof}

\section{Orthogonality and order for tripotents}

We shall recall a few standard definitions. Let $E$ be a JB$^*$-triple. By $U(E)$ we shall denote the set of all tripotents of $E$, that is the set of all elements $u\in E$ for which $u=\J uuu$. This is always nonempty subset as zero is  a tripotent. In case when $E$ is a JB$^*$-algebra,  then  tripotents are just partial isometries, that is  elements $u$ with $u=uu^*u$.  Then $p_i(u)=u^*u$ and $p_f(u)= uu^*$ are projections, called initial and final projection, respectively. 
We shall be mainly interested in the following orthogonality relation on $U(E)$. 

\begin{df} Let $E$ be a JB$^*$-triple.
	Two tripotents  $e,f\in E$ are orthogonal if 
	\[ L(e,f) = 0\,.       \]
\end{df}

In the next  proposition we shall gather basic simple characterizations of orthogonality of tripotents that we shall use in the sequel without further  comments (see e.g. Lemma~2.1 in \cite{HKP}).    

\begin{pr}\label{E4E}
	Let $E$ be a JB$^*$ triple, and let $e,f$ be tripotents in $E$. Then the following assertions are equivalent:
	\begin{enumerate}
		\item $e\perp f$
		\item $f\perp e$
		\item $e\in E_0(f)$
		\item $E_2(e)\subset E_0(f)$
		\item $\J eef=0$
		\item Both $e+f$ and $e-f$ are tripotents. 
	\end{enumerate}

\end{pr}

If $p$ and $q$ are projections in a JB$^*$-algebra $A$, then $p$ and $q$ are orthogonal as projections ($p\circ q=0$) if and only if they are orthogonal as tripotents. In case of J$^*$-algebras orthogonality is equivalent to pairwise orthogonality of initial and final projections. This relation is known under the name double orthogonality.

\begin{pr}\label{double orthogonality}
	Two tripotents $u$ and $v$ in a unital J$^*$-algebra $A$ are orthogonal if and only if they have orthogonal initial and final projections, that is
	\[         v^*u=uv^*=0\,.\]
\end{pr} 

\begin{proof}
	Suppose that $v\perp u$. This is equivalent to 
	\[   \J vvu=        0\,.  \]  
	Therefore 
	\begin{equation}\label{E1}  vv^*u+uv^*v =0\,.\end{equation}
	By multiplying from the left by $v^*$ we obtain
	\[ 0= v^*vv^*u+v^*uv^*v = v^*u+ v^*uv^*v= v^*u[1+v^*v]       \]
	
	However,  as $v^*v$ is a projection and therefore positive element,  we have that 
	$1+v^*v$ is invertible and so $v^*u=0$ by the previous identity. Multiplying now the identity \eqref{E1} by $v^*$ from the right,  we arrive similarly to 
	\[ [vv^*+1]uv^*=0  \]
	which gives $uv^*= 0$ in the same way as above. The reverse implication is obvious.  
\end{proof}

Now we shall define key order considered in this paper. 

\begin{df}
	Let $E$ be a JB$^*$-triple, and let 
	$e,f\in E$ be tripotents. We say that $e$ is less then  $f$, written $e\le f$,  if $f-e$ is a tripotent orthogonal to $e$. 
\end{df}

We gather a few characterizations of the order relation that we shall use frequently. 

\begin{pr}\label{order}
	Let $E$ be a JB$^*$-triple, and let $u,v$ be tripotents in $E$. The following assertions are equivalent (see e.g Proposition 2.4 in \cite{HKP}): 
	
	\begin{enumerate}
		\item $u\le v$
		\item $u=\J uvu$
		\item $u=\J uuv$
		\item $u$ is a projection in  $E(v)$
		\item $E(u)$ is a JB$^*$-subalgebra of $E(v)$. 
	\end{enumerate}
	
\end{pr}

Having a JB$^*$-triple $E$ we shall always consider  $U(E)$ the set of all tripotents in $E$ as a poset with  order defined above. It is a generalized orthomodular poset, where local orthocomplementation in interval $[0,e]$ is given by

\[ f^{\perp_e}=e-f\,.                     \]   

It can be observed immediately that $e\perp f$ for two tripotents $e,f$ exactly when  $e\le u-f=  f^{\perp_u}$  in each interval $[0,u]$ containing $e$ and $f$. (Such interval exists because $e+f$ is supremum $e\lor f$.)  In other words, orthogonality is induced  by order an local orthocomplementation.\\ 

Let $E$ be a JBW$^*$-triple. It is a consequence of  Proposition~\ref{order} that any interval 
$[0, e]$ in $U(E)$   is order isomorphic to the projection poset $P(E(e))$, which is  an orthomodular poset.  Moreover if $e,f$ are orthogonal tripotents and $w$ is a tripotent $w\ge e,f$,  then $e$ and $f$ become orthogonal projections in $E(w)$. (Especially this holds for $w=e+f$.) Therefore the poset $U(E)$ can be seen  as pasting orthomodular posets that are isomorphic to projection posets of JB$^*$-algebras. \\

If $A$  is a JBW$^*$-algebra,  then $P(A)$ is a complete lattice. This is far from being  true in case of tripotent order structures. The reason  is that this poset is  not  upward directed in a typical situation. To see it,  let us recall that maximal elements in $U(E)$, where $E$ is a nonzero JBW$^*$ triple,   are precisely complete tripotents,  that is,  extreme points of the unit ball. If $\dim E\ge 2$, there must be  two different maximal tripotents $u$ and $v$.  It is then straightforward to conclude that there is no upper bound of the set $\{u,v   \}$.\\

The relation of orthogonality for tripotents is an orthogonality relation in the sense of our Definition~\ref{kometa}. Further,  tripotent poset  $U(E)$ is an example of generalized quantum logic. Even if it is not a lattice,  a deep analysis given by Edward and R\"{u}ttimann in \cite{Edwards-Rut} showed that $U(E)$ is a conditionally complete lattice.  More specifically,  they showed  that  there is an order anti-isomorphisms between $U(E)$ and the set of nonempty weak$^*$ closed faces of the unit ball of $E$  ordered by set inclusion. This antiisomorhism  is  given by the map
\[  e\in U(E)\to e+B_1(E_0(e))\,.             \]  
In this light,  let us look at the lattice operation. The supremum of two elements $e,f\in U(E)$ exists if and only if the intersection of the faces $e+B_1(E_0(e))$ and  $f+B_1(E_0(e))$ is nonempty.  However, it may easily  happen that this intersection is empty. For example,  one can take two distinct extreme points $f, g$ of the unit ball of $E$ (they correspond to  maximal tripotents) and consider singleton faces $\{e\}, \{f\}$. In contrast to this, given two weak$^*$ closed  nonempty faces $E$ and $F$,  there is always their suremum,  namely  the smallest weak$^*$ closed face contains $F\cup  G$. It means that infimum $e\land  f$  in $E(U)$ always exists and corresponds to a weak$^*$-closed face generated by two faces.    \\

As  an illustration of  the triple order,  let us consider a \vNa{} $M$. Then $U(M)$ is the set of all partial isometries and the order relation is given by

\[  u\le v \text{ if and only if }     u=uv^*u=uu^*v\,.            \] 

All unitaries are maximal tripotents, however  there might be non-unitary maximal tripotents in case of infinite algebras (see \cite{HKP} for deeper analysis of this phenomenon.) \\

The following description of triple order was given in Proposition {4.6} in \cite{HKP}.\\

\begin{pr}\label{order in vNa}
	
	Let $u,v$ be partial isometries in a \vNa 
	{} $M$. The following assertions are equivalent:
	
	\begin{enumerate}
		\item $u\le v$
		\item There is a unique projection $p\le p_f(v)$ \st{} $u=pv$.
		\item  There is a unique projection $q\le p_i(v)$ \st{} $u=vq$.
	\end{enumerate}
\end{pr}

This implies that the interval $[0, v]$ in $U(E)$ is order isomorphic to interval $[0, p_f(v)]$ (and $[0, p_i(v)]$)
in the projection lattice $P(M)$.

Let $M$ be a \vNa{} acting on a \HS{} $H$. Suppose $u, w\in U(M)$. In order  to understand better the way how intervals in $U(M)$ may overlap,  we   shall describe $I=[0,u]\cap [0,w]= [0, u\land w]$.  This gives  a  spacial description of infima that is not referring to facial structure of the unit ball as Edwards and R\"{u}ttimann did in \cite{Edwards-Rut}.      By \prref{order in vNa}  we see that a tripotent $t$ is in this intersection if and only if
\[ t= pu=qw\,,     \] 
where $p$ is a projection under $p_f(u)$ and $q$ is a projection under $p_f(w)$. Multiplying the previous equation by $u^*$  from the right,   and using the fact $p_f(u)=uu^*\ge p$, we have  
\begin{equation}\label{Eheart}
p=puu^*=qwu^*\,.           
\end{equation}
However, this means that range of $p$ is contained in the range of $q$ and so $p\le q$.  By  symmetry argument  $p=q$. Therefore we have,
\[  I= \set{pu}{p\le p_f(u)\land p_f(w), pu=pw.  }     \] 

In other words, infimum $u\land w$ is of the form 
\[ u\land w = hu=hw \,,
\]
where $h= \sup\set{p\in P(M)}{p\le p_f(u)\land p_f(w), pu=pw}$.\\

This is an expression for infima in terms of the projection lattice. Let us now explore special geometric meaning of elements in $I$. Take $p\le p_f(u)$ with $pu=pw$.      By  (\ref{Eheart}) we have that

\begin{equation}\label{*E*}  p= pwu^*\,.
\end{equation}

This implies that $wu^*$  restricts to identity on $p(H)$. Indeed, let us take $\xi\in H$ with $p\xi=\xi$. 
By (\ref{*E*}) we have that
\[ \xi=pwu^* \xi\,.     \]
Suppose that $wu^*\xi\not= pwu^*\xi$. Then, obviously   
\[  \nor{pwu^*\xi}<\nor{wu^*\xi}\le \nor \xi\,.                 \]
This is a contradiction.       
Therefore $p(H)$ is an invariant subspace of $wu^*$ and this map is identity on it. By the same arguments,  $p(H)$ is  invariant for $uw^*=   (wu^*)^*$ and so $p(H)^\perp$ is also invariant for $uw^*$.  
So we obtain the following orthogonal decomposition: 

\[  uw^*= \text{identity on }p(H)\oplus \text{ some contraction on } (1-p)(H)\,.              \]


Now we turn to morphisms between tripotent posets.


\begin{exm}\label{ex1}
	Let $\Phi: E\to F$ be a Jordan triple homomorphism between JB$^*$-triples $E$ and $F$.  Then $\Phi$ preserves tripotents and restricts to  a quantum logic morphism  $\f: U(E)\to U(F)$.   
	
\end{exm}

The proof of this statement is straightforward. In the opposite direction Jordan triple homomorphisms  between JBW$^*$-triples can be characterised as  linear maps  preserving  tripotents. This is the content of the following proposition. We expect this fact to be known but we give the argument for the sake of completeness. \\

\begin{pr}\label{E*1} Let $E$ and $F$ be JBW$^*$- triples. 
	Let $\Phi: E\to F$ be a bounded  linear map  preserving tripotents. Then $\Phi$ is a Jordan triple homomorphism.
	
\end{pr}

\begin{proof}

	First we show that  $\Phi$ restricts to a quantum logic morphism  between $U(E)$ and $U(F)$.   Let $e$ and $f$ be orthogonal tripotents in $E$. Then $e+f$ and $e-f$ are tripotents in $E$, implying that $\Phi(e)+\Phi(f)$ and  $\Phi(e)-\Phi(f)$ are tripotents in $F$. By \prref{E4E} again we have that $\Phi(e)$ and $\Phi(f)$ are orthogonal. Therefore, $\Phi$ preserves  orthogonality of tripotents. This is enough for showing that $\Phi$ is a triple homomorphism. Indeed,  in view of polarization identities  (see e.g \cite{Chu})  it suffices to show that $\Phi$ preserves cubic powers. Let us take  take an element $x$ and try to show that 
	\[  \Phi(x^3)  =  \Phi(x)^3\,.       \]       
	By the spectral theorem and continuity of $\Phi$ we can suppose that \[ x= \sum_{i=1}^n \l_i e_i \,,   \] 
	where \seq en{} are orthogonal tripotents and $\seq \l n\in \cc$ . As $\Phi$ preserves orthogonality of tripotents we have that 
	\[    \J{\Phi(e_i)}{\Phi(e_j)}{\Phi(e_k)} =0                \]
	whenever $\{i,j,k\}$ is not singleton. Based on it we can compute
	\begin{multline*}  \Phi(x^3)=\Phi\biggl(\sum_{i=1}^n \l_i\overline{\l_i}\l_i  e_i\biggr) = \sum_{i=1}^n \l_i\overline{\l_i}\l_i  \Phi(e_i) =\\
	= \sum_{i,j,k=1}^n \l_i \overline{\l_j}\l_k \J{\Phi(e_i)}{\Phi(e_j)}{\Phi(e_k)}= \sum_{i=1}^n \l_i\overline{\l_i}{\l_i}\, 
	\J{\Phi(e_i)}   {\Phi(e_i)}  {\Phi(e_i)} =       \Phi(x)^3\,.                      
	\end{multline*}

\end{proof}

Therefore, in case of linear maps the situation  is clear. However, in the context of JBW$^*$-triplets duality between triple morphisms and Jordan triple morphisms breaks  down. Indeed,  next series of examples shows that order automorhisms  of tripotent poset may not be coming from restrictions of  linear maps. 
Also we demonstrate that the relationship between orthogonality and order is more delicate  for tripotents than for projections in  Jordan algebras.

\begin{exm}\label{examples}
	\vspace{2mm}

	\begin{enumerate}
		{\rm	
			\item   {\em Orthoisomorphisms not extendable to a homogeneous map:} \\
			
			Let $A$ be a JB$^*$-algebra. Then the star operation $x\to x^*$ is a bijection that preserves Jordan product,   and so its  triple product as well.  Therefore,  when restricted to $U(A)$ and taking into account  Proposition~\ref{order} and Proposition~\ref{order in vNa},  we obtain an orthoisomorphism and order isomorphism that is not extendable to any linear map. 
			
			\item {\em Orthoisomorphism not extendable to an additive map:} \\
			
			Let $M$ be a \vNa with $\dim E\ge 3$. Let $\mathbb{T}$ be the unit circle in $\cc$.  Let us introduce equivalence relation on $U(M)$ by putting $u\sim v$ if there is a complex unit $\l$ such that $u=\l v$.   	Choose an arbitrary map  $T : U(E)/\sim\,  \to \mathbb T$. Let $S: U(E)/\sim\,  \to U(E)$ be a selection function. Finally, define a   map $\f: U(M)\to U(M)$ by  $\f( \l S([u]))=  \l T([u]) S([u])$, $u\in U(E)$,  $\l\in \mathbb T$. Then $\f$  is a bijection for which $\f(S([u]) = T([u])S([u])$. Moreover,  \f{} preserves orthogonality and order in both directions. Indeed, let $u\perp v$ in  $U(E)$, or equivalently $\J uuv =0$.  There are complex units $\l$ and $\mu$ \st{}  $\f(u)=\l u$ and $\f(v)=\mu v$. Then 
			\[  \J {\f(u)}{\f(u)}{\f(v)}= \l\overline {\l} \mu \J uuv=0\,.          \]  
			Similarly,  it can be verified that \f{} preserves orthogonality  in the opposite direction. We can now specify the map $\f$ so that it has no extension to any additive map  acting on $M$. Indeed, take two nonzero orthogonal tripotents $u$ and $v$ in $M$. Modify  parameters in definition of \f{} so that $\f(u)= -u$ and $\f(v)=v$. There is a $\l\in \mathbb T$ \st{} $\f(u+v)=\l(u+v)$
			Since $\f(u)+\f(v)=v-u$,  we can see that $ \f(u+v)\not=\f(u)+\f(v)$. So \f{} cannot have an additive extension over $M$. Therefore there are orthoisomorphisms of $U(M)$ which cannot be extended to any additive map from $M$ to $M$. \\
			
			\item {\em Orthoisomorphism that is not preserving the order:   }\\
			
			We shall show  that there is a tripotent orthoisomorphism that is not preserving the order. 	
			This cannot happen in the projection poset of JB$^*$-algebras (see  \prref{poset-1}).       We shall use the preceding  example (ii). Let us take linearly independent  tripotents $u$ and $v$ in $M$ with $u\le v$. We can certainly choose the map $\f$ above  so that $\f(u)=u$ and $\f(v)=-v$. Then $u=\J uvu$. But   $\J {\f(u)}{\f(v)}{\f(u)}= -u$ and so $\f(u)$ is not underneath $\f(v)$ as $u=\f(u)\not=   \J {\f(u)}{\f(v)}{\f(u)}$.\\
			
			There is another example showing considerable nonlinearity of tripotent orthoisomorphisms.

			Let $M=B(H)_a$,  where $\dim H = 5$. By a rank, $d(u)$,  of a tripotent $u$ in $M$   we mean  dimension of it initial (and so final) projection.  First we observe that if $u\in M$  is a nonzero tripotent, then the following cases may occur (see e.g.  \cite{HKP}):  either $d(u)=2$  or $d(u)=4$.  In the former  case $u$ is minimal.  Indeed, let $e\le u$ be a nozero tripotent. Then $p_f(e)$ and $p_f(u-e)$ are orthogonal projections underneath $p_f(u)$ and with even dimensions. This immediately implies that $e=u$. In the latter case  we infer in a similar way that $u$ is a maximal tripotent.   
			Let us now consider a bijection $\f: U(M)\to U(M)$ that is fixing  tripotents with rank two  and preserves zero and tripotents of rank 4.  As nontrivial tripotents are orthogonal only if they have rank two, we can see that $\f$ is an orthoisomorphism. Let us now fix tripotents $u\le v$ such that $\dim p_i(u)=2$  and $\dim p_i(v)=4$. We can further specify $\f$ to send $v$ to a tripotent whose initial projection is not above $p_i(u)$. Then $\f(v)$ is not above $\f(u)$ and therefore $\f$ is not order preserving.

			\item {\em Orthoisomorphism on a regular JBW$^*$-triple having no linear extension.  }\\
			
			Let us now consider $M=B(H)_a$, where $\dim H=3$. Let $u$ be a nonzero tripotent of $M$. Then its rank  must be two.  Moreover, tripotents $u$ and $v$ are orthogonal if and only if at least one of them is zero.   Therefore, any bijection $\Phi$ acting on  $U(M)$ fixing zero is an orthoisomorphism. Of course, $\f$ may not have any  linear extension to $M$. On  the other hand,    $M$ is regular.  To see it we  consider a homotope $M(u)$ where $u$ is a  maximal tripotent  (i.e. $u$ has rank two). Simultaneously $u$ is an atom and so $E(u)$ is isomorphic to \cc.  
			For this reason $E(u)$ cannot contain any Type $I_2$ direct summand and so $M$ is regular. 
			
		}

	\end{enumerate}
\end{exm}

\section{Description of quantum morphisms  - consistent systems of Jordan maps}

\begin{df}
	Let $E$ and  $F$ be  JB$^*$-triples, and let $\f: U(E)\to U(F)$ be a map.     The system of maps ${\mathbf \Phi}(\f)= (\Phi_u)_{u\in U(E)}$ is a consistent system of Jordan *-homomorphisms corresponding to \f{}  if  the following conditions are satisfied:
	\begin{enumerate}
		\item  Each  $\Phi_u: E(u)\to F({\f(u)})$ is a unital Jordan *-homomorphism between algebras $E(u)$ and $F({\f(u)})$. 
		\item If $u\le v$ in $U(E)$ then $\Phi_u$ and $\Phi_v$ coincide on $E(u)$.  
	\end{enumerate}

\end{df}

Remark that the map \f{} in the above definition is uniquely determined by the system of maps $(\Phi_u)_{u\in (U(E))}$ as 
\[  \f(u)=\Phi_u(u)\,.            \] 
On the other hand, to a given \f{} there is only one possible consistent system ${\mathbf \Phi}(\f)$.
Indeed, suppose we have two such consistent systems ${\mathbf \Phi(\f)}$ and ${\mathbf \Psi(\f)}$. Consider a tripotent $u\in E$. We know that $\Phi_u$ and  $\Psi_u$ coincide with \f{} on $[0,u]=P(E(u))$. However projections in $E(u)$ span a dense linear subspace. So by continuity $\Phi_u=\Psi_u$.

\begin{pr}\label{E**}
	Let $ {\mathbf \Phi}(\f)$ be a consistent system of Jordan *-homomorphisms  between JB$^*$-triples $E$ and $F$. Then the map \f{} is a quantum logic morphism.
\end{pr}

\begin{proof}
	Let $u$ and $v$ be orthogonal tripotents. Then $w=u+v=u\lor v$ is a tripotent. By the assumption $\Phi_w$ gives $\Phi_u$ and $\Phi_v$ on $E(u)$ and $E(v)$, respectively. Tripotents $u$ and $v$ becomes projections in $E(w)$ and are mapped  to orthogonal projections $\Phi_w(u)=\f(u)$ and $\Phi_w(v)=\f(v)$ as $\Phi_w$ is a Jordan *- homomorphism.    Moreover,  we have 
	\[ \f(u\lor v)=\f(u+v)=\Phi_w(u+v)= \Phi_w(u)+\Phi_w(v)= \f(u)+\f(v)=\f(u)\lor \f(v)\,.                      \]

\end{proof}

\begin{df}
	Let $E$ and  $F$ be  JB$^*$-triples.  The consistent system of Jordan *-homomorphisms ${\mathbf \Phi}(\f)$         is called a consistent system of Jordan *-isomorphisms  if each map $\Phi_u$ is a Jordan *-isomorphism and \f{ } is a bijection.

\end{df}

\begin{pr}
	Let ${\mathbf \Phi}(\f)$ be a consistent system of Jordan *-isomorphisms. Then $\f$ is a quantum logic isomorphism.  
	
\end{pr}

\begin{proof}
	By Proposition~\ref{E**} the corresponding map \f{} is a bijection that is a quantum logic isomorphism. Let us now realize that the system of maps $(\Phi_w^{-1})_{w\in U(F)}$ it is a consistent system of Jordan *-homomorphism corresponding to $\f^{-1}$. For this reason $\f^{-1}$ is a quantum logic morphism as well.


	
\end{proof}

Now we prove that any quantum logic morphism   between structure of tripotents  extends uniquely to a consistent system of Jordan maps.

\begin{theo}\label{extension}
	Let $E$ and $F$ be JBW$^*$-triples, where $E$ is regular. Let  $\f: U(E)\to U(F)$ be a   quantum logic morphism.   Then   there is a unique consistent system of Jordan *-homomorphisms   ${\mathbf \Phi}(\f)=(\Phi_u)_{u\in U(E) } $ corresponding to \f{}.

\end{theo}

\begin{proof}
	We know that \f{} preserves the order.   	Let us fix $u\in U(E)$ and take a complete tripotent $w\in E$ with $w\ge u$.  Then \f{} maps $[0,u]=P(E(u))$ into $[0,\f(u)]= P(E({\f(u)}))$  and $[0,w]=P(E(w))$ into $[0,\f(w)]= P(E({\f(w)})).$ 
	According to Theorem~\ref{Dye-Jordan algebras} there is a unital Jordan *-homomorphism $\Phi_w: E(w)\to F(\phi(w))$. 
	As $E(u)$ is a JB$^*$-subalgebra of $E(w)$, $\Phi_w$ restricts to a Jordan *-homomorphism $\Phi_u$  with domain $E(u)$ whose restriction to $[0,u]$ coincides with  \f{}. Such a map is unique and so does not depend on the choice of $w$. 
	Therefore, we  have that for each $u\in U(E)$ the restricted map
	$\f: [0,u]\to [0,\f(u)])$ extends uniquely to   a unital Jordan $*$-homomorphism   $\Phi_u$   between $E(u)$ and $F({\f(u)})$.  It remains to verify that ${\mathbf{\Phi}}(\f)$  is a consistent system of Jordan *-homomorphisms.
	But this follows immediately from the construction.

\end{proof}

\begin{theo}\label{isomorphism}
	Let $E$ be a regular JBW$^*$-triple and $F$ a JBW$^*$-triple. Let $\f: U(E)\to U(F)$ be a quantum logic  isomorphism. Then there is a unique consistent system ${\mathbf \Phi}(\f)$       of Jordan *-isomorphisms. 
	
\end{theo}

\begin{proof}
	We know that there is a consistent system of Jordan *-homomorphisms   ${\mathbf \Phi}(\f)$  with corresponding function \f. We have to show that each $\Phi_u$ is a Jordan *-isomorphism. Fix $u\in U(E)$. Then 
	the restriction $\f: E(u)\to E(\f(u))$ is a quantum logic isomorphism. So by \theoref{Dye-Jordan injective} we have that $\Phi_u$ is a Jordan *-isomorphisms.
	
\end{proof}

\section{Local  Jordan  morphisms}

In this part we describe morphisms of tripotent posets using  one single map rather than  a family of Jordan maps. It turns  out that  this global map is partially linear  in the sense of definitions below.

\begin{df}
	We say that a set $S$ in a JBW$^*$-triple $E$ is triple bounded if the set $\set{r(s)}{s\in S}$ has  upper bound in $U(E) $.

\end{df}

\begin{df}
	Let $E$ and $F$ be JBW$^*$-triples. Let $J:E\to F$ be a map. We say that $J$ is a local triple Jordan  homomorphism if the following conditions are satisfied. 
	\begin{enumerate}
		\item $J$ is real  homogenous,  i.e. $J(\l x)=\l J(x)$ for all $x\in E$ and $\l\in \rr$.  
		\item $J$ is partially additive in the sense   $J(x+y)= J(x)+J(y)$,   whenever the set $\{x,y\}$ is triple bounded. 
		\item $J$ preserves tripotents.  
	\end{enumerate} 
\end{df}

First we observe that in view of Proposition~\ref{E*1}  any  linear  local Jordan morphism is a Jordan triple homomorphism. Further we see that any local Jordan isomorphim restricts to a quantum logic morphism preserving reflections.

\begin{pr}\label{T1}
	Suppose that $E$ and $F$ are  JBW$^*$-triples.         Let $J:E\to F$ be a local triple Jordan  homomorphism between JBW$^*$-triples $E$ and $F$. Then $J$  restricts to quantum logic morphism $\f: E\to F$ such that $\f(-u)=-\f(u)$ for all $u\in U(E)$.

\end{pr}

\begin{proof}

	Suppose  that $e$ and $f$ are orthogonal tripotents. Tripotent  $e+f$ and $e-f$ is supremum of the set $\{e,f\} $ and $\{e,-f\}$, respectively. Therefore the sets 
	$\{e,f\} $ and $\{e,-f\} $ are triple bounded.    
	By assumption  $J(e+f)=J(e)+J(f)$ and  $J(e-f)=J(e)-J(f)$. As $e\pm f$ are tripotents, we conclude that $J(e)\pm J(f)$ are tripotents as well. This shows that $J(e)$ and $J(f)$ are orthogonal tripotents. Moreover we see that $\f(e\lor f)=\f(e)\lor \f(f)$.  
	The proof is completed.



\end{proof}

In order to prove the opposite statement    we shall need the following auxiliary lemma.

\begin{lem}\label{range projection}
	Let $E$ be a JBW$^*$-triple. Let  $w\in U(E)$.  The following statements hold: 
	\begin{enumerate}
		\item If $x$ is a positive element in $E(w)$, then $r(x)\le w$.
		\item Let $\seq x n \in E$ be such that $r(x_i)\le w$ for all $i=1,\ldots,  n$,  then $r(\sum_{i=1}^n x_i)\le w$. 
	\end{enumerate}
\end{lem}

\begin{proof}
	(i) Let $x$ be a positive element in $E(w)$.  As $E(w)$ is a JBW$^*$-subtriple of $E$, we have that the range tripotent $r(x)$ belongs to $E(w)$.      However, $r(x)$ is the smallest tripotent such that $x$ is positive in $E(r(x))$. Hence, $r(x)\le w$.         \\
	
	(ii) We have that  $x_i$ is a positive element in $E(r(x_i))$ and so also in $E(w)$ because $E(r(x_i))$ is a *-subalgebra of $E(w)$. For this reason $\sum_{i=1}^n x_i$ is positive in $E(w)$ and (i) applies. \\

\end{proof}

Let us remark that the previous proposition does not hold without assuming positivity of $x$. Indeed, any tripotent in $E(w)$ that is not a projection provides a counterexample.  

\begin{theo}\label{extension local}
	Let $E$ and $F$ be JBW$^*$-triples. Suppose that $E$ is regular.  Let 
	\[  \f :U(E)\to U(F)\,.\] Then the following conditions are equivalent:
	\begin{enumerate}
		\item \f{} is a quantum logic morphism such that $\f(-u)=-\f(u)$ for all $u \in U(E)$.  
		\item 	   There is a local Jordan triple morphism
		\[ \Phi: E\to F\]
		extending \f.
		\item  	There is a  local Jordan triple morphism
		\[ \Phi: E\to F\]
		extending \f{}  such that moreover 
		\[ \Phi\J xyx = \J{\Phi(x)}{\Phi(y)}{\Phi(x)} \,,      \] 
		whenever the set $\{x,y\}$ is triple bounded.
	\end{enumerate}  
Further, if  $\f:E\to F$ is a quantum logic morphism,  then it extends to a partially additive map between $E$ and $F$. 
\end{theo}

.
\begin{proof}
	(i)$\Rightarrow$ (iii).  Let $\f$ be a quantum logic morphism.	By \theoref{extension} we have a consistent system of Jordan *-homomorphisms ${\mathbf\Phi}( \f)$ corresponding to \f. 
	Define now the map $\Phi$  in  the following way: Let $x\in E$,  set 
	\[   \Phi(x)=\Phi_{r(x)}(x)\,.      \]	
	Let us verify that $\Phi$ is a local Jordan  triple morphism.  Let us fix $x\in E$ and $\l\in \rr$. If $\l>0$ then it can be seen readily $r(x)=r(\l x)$. Therefore,
	
	   \[   \Phi(\l x)=\Phi_{r(\l x)}(\l x)=\l \Phi_{r(x)}(x)= \l \Phi(x)\,.         \]
Let us discuss the case  when $\l<0$. It holds that  $r(-x)=-r(x)$. Indeed,  for any tripotent $u\in E$ we have that $E_2(u)=E_2(-u)$ and  $x^{*_u}=x^{*_{-u}}$, $x\circ_u y = - x\circ_{-u} y$ for all $x,y\in E_2(u)$.   This gives that an element $z\in E_2(u)$ is positive in $E(u)$ if and only if $-z$ is positive in $E(-u)$. Therefore $-x$ is positive in $E(-r(x))  	$. On the other hand,  suppose that $-x$ is positive in $E(u)$ for a tripotent $u$. By the above $x$ is positive in $E(-u)$ and so $r(x)\le -u$. Equivalently, $r(x)= \J{r(x)}{-u}{r(x)}$ and so $-r(x)= \J{-r(x)} {u} {-r(x)}$. This means that $-r(x)\le u$. Hence, $r(-x)=-r(x)$. Further,  we shall need the fact that $\Phi_u=\Phi_{-u}$.  Let us remark that by the assumption both $\Phi_u$ and $\Phi_{-u}$  act between $E_2(-u)$ and $E_2(\f(-u))$. Let us take a projection $p\in E(-u)$. Then $-p$ is a projection in $E(u)$. Indeed, the *-operation is the same and so $p$ is self-adjoint in $E(u)$. The idempotency of $p$ in $E(-u)$ means that 
$p = \J p{-u}p$. Then    $(-p)\circ_u (-p)= \J pup = -p$.   Therefore, $-p$ is a projection in $E(u)$. 
In fact, by symmetry, $q$ is a projection in $E(u)$ if and only if $-q$ is projection in $E(-u)$  Let us fix a projection $p\in E(-u)$. Then

\[ \Phi_u(p) = - \Phi_u(-p)=-\f(-p)=  \f(p).                      \]
It means that $\Phi_u$ and $\Phi_{-u}$ coincide on projections in $E(-u)$.  Moreover,  $\Phi_u$  preserves Jordan product in $E(-u)$.  For this, let us  take $x,y\in E(-u)$ and compute
\begin{multline*}  \Phi_u(x\, \circ_{-u}\,  y)= \Phi_u(-x\, \circ_u\,  y)=
 - \Phi_u(x)\, \circ_{\f(u)}\,  \Phi_u(y)= \\ =   \Phi_u(x)\, \circ_{-\f(u)}\,  \Phi_u(y) = \Phi_u(x)\, \circ_{\f(-u)}\,  \Phi_u(y) \,.           \end{multline*}
Since any Jordan *-homomorphism is uniquely deteremined by its value on projections, we have that really $\Phi_u=\Phi_{-u}$. Finally, we can compute,   
\[ \Phi(\l x) = \Phi_{r(\l x)}(\l x)= \Phi_{-r(x)}(\l x)=  \Phi_{r(x)} (\l x)=\l \Phi_{r(x)}(x)= \l \Phi(x)\,.                               \]  We have shown that $\Phi$ is  real homeogeneous.   \\

	
	Let us investigate additivity of $\Phi$.  
	 Take  a triple bounded set $\{x,y\}$. Let $w$ be a tripotent with $w\ge r(x), r(y)$. By Lemma~\ref{range projection} we have $r(x+y)\le w$. Using consistency of the system of Jordan *- homomorphisms   ${\mathbf\Phi}( \f)$    we have 
	
	\begin{multline}
	\Phi(x+y)= \Phi_{r(x+y)}(x+y)=\Phi_w(x+y)=\\ =\Phi_w(x)+\Phi_w(y)= \Phi_{r(x)}(x)+ \Phi_{r(y)}(y)=\Phi(x)+\Phi(y)\,. 
	\end{multline}
	This shows additivity of $\Phi$ on triple bounded sets.\\
	
	Suppose now that $\{x, y\}$ is a triple bounded    set.  That is,  there is a tripotent $w$ with $r(x), r(y)\le w$.     As $x$ and $y$ are positive in $E(r(x))$ and $E(r(y))$, respectively, we can see that $x$ and $y$ are  positive in $E(w)$ as well.    Observe that $U_x(y) = \J xyx$ is positive in $E(w)$. Using Lemma~\ref{range projection} once again, we obtain that $r(\J xyx)\le w$. Now we can compute, using the properties of  a consistent system of Jordan maps, that

	\begin{multline*}      \Phi\J xyx =   \Phi_{r(\J xyx)} \J xyx  =\Phi_w \J xyx = \\ = \J {\Phi_w(x)}   {\Phi_w(y)} {\Phi_w(y)}=\J{\Phi(x)}{\Phi(y)}{\Phi(x)} \,.      \end{multline*}   
	The proof of the present implication is completed. \\

	(ii)$\Rightarrow$(i) follows immediately from \prref{T1} and (iii)$\Rightarrow$ (ii) is trivial.\\
	
	
	The fact that any quantum logic morphism $\f: U(E)\to U(F)$ extends to a partially additive map $\Phi: U(E)\to U(F)$ is contained in the proof of the implication (i)$\Rightarrow$ (iii). \\
	
	The proof is completed.

\end{proof}

\begin{df}
	Let $E$ and $F$ be JBW$^*$-triples.      A map  $J: E\to F$ is a local Jordan triple isomorphism if it is a bijection such that both $J$ and $J^{-1}$ are  local Jordan triple homomorphisms.

\end{df}

\begin{lem}\label{l2}
	Let $J: E\to F$ be a local Jordan triple isomorphism. Then its restriction \f{} to $U(E)$ is a quantum logic isomorphism  between $U(E)$ and $U(F)$ such that $\f(-u) = -\f(u)$ for all $u\in U(E)$.  	
\end{lem}   	

\begin{proof}
	It follows immediately  from  \prref{T1} that \f{} is a quantum logic isomorphism.

\end{proof}

\begin{theo}\label{Evi2}
	Let $E$ and $F$ be JBW$^*$-triples. Suppose that $E$ and $F$  are  regular. Let  $\f: U(E)\to U(F)$ be a
	quantum logic isomorphism such that $\f(-u)=-\f(u)$ for all $u\in U(E)$. Then \f{} extends to a local 
  	\end{theo}

  	\begin{proof}
  		Let us denote by ${\mathbf \Phi}$ and $\mathbf{\Psi}$ local Jordan triple homomorphisms corresponding to $\f$ and $\f^{-1}$ respectively, as constructed in the proof of \theoref{extension local}. We shall prove that  they are mutually inverse maps. 
  		For this let us take $x\in E$ and consider 
  		
  		 \[  y=\Phi(x)=\Phi_{r(x)}(x)\,.             \]

     We know that $\Phi_{r(x)}$ is a Jordan *-isomorphism from $E(r(x))$ onto $E(\f(r(x)))$ whose inverse is the map $\Psi_{\f(r(x))}$. As $x$ is positive in $E(r(x))$,  we can see that  $y=\Phi(x)$ is positive in $E(\f(r(x))$. Hence, $r(y)\le \f(r(x))$. 
      Now we can compute
      \[ \Psi(y)=  \Psi_{\f(r(x))}(y)= \Phi_{r(x)}^{-1}(y)= x\,.                    \] 
  		So we have established that 
      $\Psi\circ \Phi$ is identity on $E$. By the symmetry we have that $\Phi$ and $\Psi$ are really mutually inverse maps. Now by \theoref{extension local} $\Phi$ is a local Jordan triple  isomorphism.                        		
  		

  	\end{proof}

  	
  		
  	
  	\section{Tripotent posets as invariants}

  	   If we have two JB$^*$-triples $E$ and $F$ that are Jordan triple isomorphic,  then  the tripotent structures $U(E)$ and $U(F)$ are isomorphic as generalized quantum logics. This holds because of the fact that Jordan triple isomorphism implements quantum logic isomorphism.  Therefore the tripotent poset is invariant in the theory of Jordan triples.    
  	     We know that morphism between tripotent structures is not extendable to a Jordan triple morphism in all cases.  However, it does not exclude that tripotent structure determines the triple structure itself. It happens in case of projection lattices. Indeed, even if  Dye's theorem does not hold for type $I_2$ \vNa s,  we  still have that such algebras are Jordan *-isomorphic    if and only if their projection lattices are orthoisomorphic.  We  have  the following result proved in \cite[Cor. 9.2.9, p. 193]{Lindenhovius} and \cite[Theorem 2.3]{IQSA08} 
  	      
  	   \begin{pr}\label{E111}
  	   	Let $M$ and $N$ be \vNa s.  Then the following conditions are equivalent:
  	   	\begin{enumerate}
  	   		\item $P(M)$ and $P(N)$ are orthoisomorphic.
  	   		\item $M$ and $N$ are Jordan isomorphic (that is there is a  Jordan *-isomorphism between $M$ and $N$).
  	   		 	   	\end{enumerate}
     		 	   	\end{pr}
  	   	
  	  Therefore, projection lattice with orthogonality relation  is a complete Jordan invariant for \vNa s.      It is natural to ask whether the same holds for tripotent posets.  We shall give an affirmative answer below. Let us first state auxiliary facts.  
  	  The following one  is well known and we state it for the sake of completeness.  
  	  
  	  \begin{lem}\label{L1}
  	  	Let $A$ be a J$^*$-algebra in  $B(H)$. Suppose that $v\in U(A)$. Then $A(v)$ is triple isomorphic to $A(p_i(v))$.
  	  	 	  	
  	  	\end{lem}

    	\begin{proof}
    		We know that $A(v)=p_f(v)Ap_i(v)$ and $A(p_i(v)) = p_i(v)Ap_i(v)$. It can be easily  verified by a direct computation  that the map 
    		\[  x\in A(v) \to   v^* x\in A(p_i(v))                   \]  
    		is a Jordan triple isomorphism.

    	\end{proof}

  	  We  show that for unital \Ca s   the structure of tripotents determines the structure of projections. It is based on the following lemma.
  	  
  	  \begin{lem} \label{Ev1}
  	  	Let $A$ be a unital \Ca{} and $u$ a complete tripotent in $A$. Then $P(A)$ is quantum logic isomorphic to  the   interval $[0,u]$ in $U(A)$.  
  	  	  	  \end{lem}
  	  
   	  \begin{proof}
   	  By Lemma~6.1 in \cite{HKPP} there is a \HS{} $H$ and an isometric  unital Jordan *- homomorphism $\psi: A\to B(H)$ \st{} $\psi(u)^*\psi(u)=1$. Therefore
       we can replace $A$ by JC$^*$-algebra $\psi(A)$ that is Jordan *-isomorphic to $A$.  	  
   	   This way  we can suppose that $p_i(u)=1$. By \lemref{L1} there si a unital triple isomorphism  between      $A(u)$ and $A(p_i(u))=A(1)=A$. Therefore $A(u)$ and $A(1)$ are  isomorphic as JB$^*$-algebras. Consequently, $P(A)=[0,1]$ is quantum logic isomorphic   to $[0,u]$ in $A(u)$. The proof is completed.

   	  \end{proof}

  	    \begin{pr}\label{Nov1}
  	    	Let $A$ and $B$ be unital \Ca s such that $U(A)$ and $U(B)$ are quantum logic isomorphic. Then  $P(A)$ and $P(B)$ are quantum logic isomorphic.    	    \end{pr} 
  	  
  	  \begin{proof}
  	  	Let $\f: U(A)\to U(B)$ be an orthoisomorphism. 
  	  	Then $u=\f(1)$ is  a complete tripotent in $B$.
  	  	  	  	 Indeed, if it is not true, then there is a nonzero tripotent $h$ in $B$ orthogonal to $\f(1)$. Then the preimage $\f^{-1}(h)$ is a nonzero tripotent orthogonal to 1, which is not possible. Now \f{} restricts to an orthoisomorphism between $[0,1]$ and $[0,u].$  These posets are quantum logic isomorphic  to 
  	  	$P(A)$ and $P(B)$,  respectively by \lemref{Ev1}.

  	    \end{proof} 
  	     
    Since it is  known that projection lattice  as a quantum  logic is a complete Jordan invariant for \vNa s (see \prref{E111}) we can conclude that the same holds for tripotent poset.

  	  	\begin{theo}
  	  Let $M$ and $N$ be \vNa s. Suppose that $U(M)$ and $U(N)$ are quantum logic  isomorphic. Then $M$ and $N$ are Jordan *-isomorphic. 	
  	  	   	  \end{theo}

  	  	As a conclusion, even if the tripotent poset is larger than projection poset, it contains the same amount of information about Jordan parts of \vNa s as their  projection lattices. \\

  	{\bf   Acknowledgement:}\\
  	  
 	  This work was supported by the project OPVVV CAAS\\
 	   CZ.02.1.01/0.0/0.0/16\_019/0000778



\begin{thebibliography}{99}
  	
  	\bibitem{BW1}: L. J. Bunce and J. D. M. Wright, \textit{ Quantum measures and states on Jordan algebras.} Communications in  Mathematical Physics, {\bf 98}, 187-202 (1985), 187-202.  
  	
  	\bibitem{BW2} L. J. Bunce and J.D.M. Wright, \textit{ Continuity and linear extensions of quantum measures on Jordan operator algebras.} Math. Scand. {\bf 64} (1989), 300-306. 
  	
  	
  	\bibitem{Bunce3} L. J. Bunce and  J. D. M. Wright, \textit{ On Dye theorem for Jordan operator algebras.}
  	Expo Math. \textbf{11}, (1993), 91--95.
  	
  	\bibitem{Friedmann} Y. Friedmann, \textit{ Physical Applications of Homogeneous Balls.} Birkh\"{a}user, 2005.  
  	
  	\bibitem{Chu} Cho-Ho Chu,  \textit{Jordan Structures in Geometry and Analysis.}  Cambridge University Press, 2012. 
  	
  	\bibitem{Garcia}
  	M. Macabre Garcia and  A.  Rodriguez Palacois, \textit{Non-associative Normed Algebras.} Vol. 2, Encyklopedia of Mathematics and its Applications, Vol. 167, Cambridge University Press, Cambridge,  2018.
  	
  	
  	
  	\bibitem{Edwards-Rut}
  	C.M.Edwards and  G.T.R\"{u}timann, \textit{On the facial structure of the unit ball in a JBW$^*$ triple.} Math. Scand. \textbf{82} (2) (1989), 317-332. 
  	
  	\bibitem{Edwards-Rut2}  C. M. Edwards and  G. T. R\"{u}timann. \textit {Exposed faces of the unit ball in a JBW$^*$-triple.}  Math. Scand. \textbf{82} (1998), 287-304. 
  	
  	\bibitem{Dye} H. A. Dye, \textit{ On the geometry of projections in certain operator algebras.}  Ann. Math. {\bf 61}, (1955), 73-89.
  	
  	\bibitem{qmt} J. Hamhalter, \textit{ Quantum Measure Theory.} Kluwer Academic Publishers, Dordrecht, Boston, London (2003).
  	
  	\bibitem{Ham-Dye} J. Hamhalter, \textit{ Dye's theorem and Gleason's theorem for AW$^*$-algebras.}  Journal of Mathematical Analysis and its Applications, \textbf{422} (2015), 1103-1115. 
  	
  	\bibitem{IQSA08}  J. Hamhalter and E. Turilova. \textit{Jordan invariants of \vNa s and Choquet order on state spaces.} International Journal of Theoretical Physics, https://doi.org /10.1007/S10773-019-04157-w.  
  	
  	\bibitem{HKP} J. Hamhalter, O. F. Kalenda, A. Peralta: \textit{Finite tripotents and finite JBW$^*$ triples.} Journal of Mathematical Analysis and its Applications, \textbf{490} (1), (2020),  124217.  		
  	
  	\bibitem{HKPP} J. Hamhalter, O. F. Kalenda, A. Peralta, H. Pfitzner, \textit{ Measures of weak noncompactness in preduals of von Neumann algebras and JBW$^*$-triples.} Journal of Functional Analysis \textbf{278} (2020), 108300.     
  	
  	\bibitem{Kadison} R. V. Kadison and  J. R. Ringrose, \textit{Fundamentals of the Theory of Operator Algebras.} Academic Press, 1993. 
  	
  	
  	\bibitem{Landsman} K. Landsman, \textit{ Foundations of Quantum Theory, From Classical Concepts to Operator Algebras.}  Springer Verlag, 2017.
  	
  	\bibitem{Lindenhovius} B. Lindenhovius,  $C(A)$. PhD. Thesis, Radbound  University Nijmegen, Netherlands, 2016.  
  	
  	
  	
  	\bibitem{Molnar} L. Molnar: \textit{Selected Preservers Problems on Algebraic Structures of Linear Operators and Function Spaces.} Lecture Notes in Mathematics, Springer 2007. 
  	
  	\bibitem{Neumann-F} J. von Neumann, \textit{ Mathematical Foundations of Quantum Mechanics.} Series: Princeton Landmarks in Mathematics and Physics, 53,     Princeton University Press,   
  	2018.     
  	
  	
  	\bibitem{vN-projective geometry} J. von Neumann, \textit{ Continuous Geometry.} Princeton, 1937.
  	
  	
  	
  	
  	
  	\bibitem{Hanche}
  	H.Hanche-Olsen and  E. Stormer,  \textit{Jordan Operator Algebras},
  	Pitman Advanced Publish Program, Boston, London, Melbourne, 1984.
  	
  	
  	
  	
  	
  	
  	\bibitem{Tak} M.Takesaki, \textit{ Theory of Operator Algebras I.} Springer Verlag, 1979. 
  	
  	\bibitem{Topping}
  	D. M. Topping,  \textit{Jordan algebras of self-adjoint operators}, Mem. Am. Math. Soc., {\bf 53}, (1965).
  	
  	
  	\bibitem{Ulhorn} U.Ulhorn,  \textit{Representation of symmetry transformations in quantum mechanics.} Arkiv Fyzik, \textbf{23}  (1962), 307-340.
  	
  	
  	
  	
  	\bibitem{Upmeier} H. Upmeier,  \textit{
  		Symmetric Banach manifolds and Jordan C*-algebras.}  Math. Studies, \textbf{104}, Nort-Holland (1985)
  	
  	
  	
  	\bibitem{Wigner} E .P. Wigner, \textit{ Group Theory and Its Applications to the Quantum Theory of
  		Atomic Spectra.} 	Academic Press Inc., New York, 1959.
  	
  	
  	\bibitem{Wright}	J. D. M. Wright, \textit{  Jordan C$^*$ algebras.} Michigan Math. J. \textbf{24}  (1977), 291-302.  
  	
  \end{thebibliography}
\end{document}